\documentclass[a4paper,10pt]{amsart}
\usepackage{cmap}
\usepackage[T1]{fontenc}
\usepackage[utf8]{inputenc}
\usepackage{lmodern}
\usepackage{microtype}
\usepackage[margin=27mm]{geometry}
\usepackage{mathtools,amssymb}
\usepackage{enumitem}
\usepackage{needspace}
\usepackage{xcolor}
\usepackage[hidelinks]{hyperref}
\newif\ifshowchanges
\showchangestrue

\DeclareMathOperator{\Ann}{Ann}
\DeclareMathOperator{\Hom}{Hom}
\DeclareMathOperator{\Ext}{Ext}
\DeclareMathOperator{\Spec}{Spec}
\DeclareMathOperator{\Max}{Max}
\DeclareMathOperator{\Att}{Att}
\DeclareMathOperator{\tor}{tor}
\DeclareMathOperator{\Supp}{Supp}
\DeclareMathOperator{\Cl}{Cl}
\newcommand{\GV}{\operatorname{GV}}
\newcommand{\wMax}{w\text{-}\Max}
\newcommand{\tMax}{t\text{-}\Max}
\newcommand{\wSpec}{w\text{-}\Spec}
\newcommand{\wSupp}{w\text{-}\Supp}
\newcommand{\wdim}{w\text{-}\dim}
\newcommand{\Nag}{\{X\}}
\newcommand{\N}{\mathbb N}
\newcommand{\Q}{\mathbb Q}
\newcommand{\Z}{\mathbb Z}
\newcommand{\wrad}[1]{\sqrt[w]{#1}}
\newtheorem{theorem}{Theorem}[section]
\newtheorem{proposition}[theorem]{Proposition}
\newtheorem{lemma}[theorem]{Lemma}
\newtheorem{corollary}[theorem]{Corollary}
\theoremstyle{definition}
\newtheorem{definition}[theorem]{Definition}
\theoremstyle{remark}
\newtheorem{remark}[theorem]{Remark}
\setlist[enumerate]{label=\textnormal{(\arabic*)},leftmargin=2.2em,itemsep=2pt}

\begin{document}
\title[$t$-almost Dedekind domains and $w$-secondary modules]
{Characterizations of $t$-Almost Dedekind Domains via $w$-Secondary Modules}
\author[H. Kim]{Hwankoo Kim}
\address{Division of Computer Engineering, Hoseo University, Asan 31499, Republic of Korea}
\email{hkkim@hoseo.edu}
\author[H. Baek]{Hyungtae Baek}
\address{School of Mathematics, Kyungpook National University, Daegu 41566, Republic of Korea}
\email{htbaek5@gmail.com}
\author[H. Hamdi]{Haleh Hamdi}
\address{CEMAT and Departamento de Matem\'atica, Faculdade de Ci\^encias da Universidade de Lisboa, Campo Grande, Edif\'icio C6, 1749-016 Lisboa, Portugal}
\email{hhamdi@fc.ul.pt; h.hamdimoghadam@tabrizu.ac.ir}
\thanks{H. Hamdi is the corresponding author.}
\thanks{H. Kim was supported by the Basic Science Research Program through the National Research Foundation of Korea (NRF), funded by the Ministry of Education (2021R1I1A3047469). H. Hamdi was supported by Funda\c{c}\~ao para a Ci\^encia e Tecnologia, Portugal, through projects \url{https://doi.org/10.54499/UID/04621/2025} and \url{https://doi.org/10.54499/2023.06156.CEECIND/CP2831/CT0002}.}
\thanks{ORCID iDs: H. Kim, 0000-0002-3563-077X; H. Baek, 0000-0003-1477-5715; H. Hamdi, 0000-0002-1465-5679.}
\date{}

\begin{abstract}
Let $R$ be an integral domain and let $w$ denote the $w$-operation on $R$. A nonzero GV-torsion-free $R$-module $M$ is called $w$-secondary if, for every $r\in R$, either $M/rM$ is GV-torsion or $r^nM=0$ for some $n\in\N$. A $w$-module $M$ is said to admit a $w$-secondary representation if $M=(N_1+\cdots+N_n)_w$, where each $N_i$ is a $w$-secondary submodule of $M$. Such a representation is called minimal if the ideals $P_i=\sqrt{\Ann_R(N_i)}$ are pairwise distinct, and $M$ is called $w$-representable if it admits such a representation. If the sum can be chosen direct, then $M$ is called strongly $w$-representable. We characterize $w$-secondary modules and investigate their localization and $w$-closure properties. For a nonfield domain $R$, we prove that $R$ is $t$-almost Dedekind if and only if $\wdim R=1$ and every $w$-representable $w$-module is locally strongly $w$-representable: each of its nonzero localizations at maximal $w$-ideals is strongly representable. Global strong $w$-representability is retained when all the radical annihilators in the representation are nonzero.
\end{abstract}

\maketitle

\makeatletter
\renewcommand\@makefnmark%
{\mbox{\textsuperscript{\normalfont\@thefnmark)}}}
\makeatother

\section{Introduction}
As a dual notion to the primary decomposition of modules over a commutative ring, the concept of secondary representation has been studied since the seminal work \cite{Macdonald1973}. Recall that a nonzero unital module $M$ over a commutative ring $R$ with identity is said to be {\em secondary} if, for each $r\in R$, either $rM=M$ or $r^nM=0$ for some $n\in\N$.

In recent years, the theory of secondary representations of modules over commutative rings with identity has attracted considerable attention; see, for instance, \cite{Atani2001,Pakyari2020,PakyariRMJ2023,PakyariCA2023}, especially in the context of Dedekind and almost Dedekind domains. The notion has also been extended with respect to multiplicatively closed subsets and torsion theories; see \cite{Farshadifar2021,Garcia2013,Rezazadeh2026}.

Let $\tau$ be an idempotent kernel functor in the category $R$-Mod, and let $\mathcal T_\tau$ and $\mathcal F_\tau$ denote the classes of $\tau$-torsion and $\tau$-torsion-free modules. In \cite{Garcia2013}, an $R$-module $M$ is called {\em $\tau$-secondary} if $M\notin\mathcal T_\tau$ and, for every $r\in R$, multiplication by $r$ has torsion cokernel or $r^nM\in\mathcal T_\tau$ for some $n\in\N$. We investigate secondary representations with respect to the $w$-operation, focusing on $t$-almost Dedekind domains.

Let $R$ be an integral domain with quotient field $K$. For nonzero fractional ideals $I,J$ of $R$, set $(I:_KJ)=\{x\in K\mid xJ\subseteq I\}$ and $I^{-1}=(R:_KI)$. The $v$-, $t$-, and $w$-closures are
$$
 I_v=(I^{-1})^{-1},\qquad
 I_t=\bigcup_{\substack{0\neq J\subseteq I\\J\text{ finitely generated}}}J_v,
 \qquad
 I_w=\bigcup_{\substack{J\subseteq R\text{ finitely generated}\\J^{-1}=R}}(I:_KJ).
$$
For $\star\in\{v,t,w\}$, these operations satisfy $(xI)_\star=xI_\star$, $R_\star=R$, $I\subseteq I_\star$, monotonicity, and idempotence. An ideal $I$ is a {\em $\star$-ideal} if $I=I_\star$, and a maximal $\star$-ideal is maximal among the proper integral $\star$-ideals. The ideal $I$ is {\em $\star$-invertible} if $(IJ)_\star=R$ for some fractional ideal $J$. One has $\tMax(R)=\wMax(R)$, and maximal $w$-ideals are prime \cite[Corollary~2.17]{AndersonCook2000}. Also, $t$-invertibility and $w$-invertibility coincide. The {\em $t$-dimension} is the supremum of the lengths of chains of prime $t$-ideals. Similarly, {\em $\wdim R$} is the supremum of the lengths of chains of prime $w$-ideals, equivalently the supremum of their heights. Thus $t\text{-}\dim R\leq\wdim R\leq\dim R$.

A finitely generated ideal $J$ of a commutative ring $R$ is a \emph{Glaz--Vasconcelos ideal}, or \emph{GV-ideal}, if the natural homomorphism $R\to\Hom_R(J,R)$ is an isomorphism; write $J\in\GV(R)$. Over a domain, $\Hom_R(J,R)\cong J^{-1}$. For an ideal of a commutative ring, the same $w$-closure is described by $I_w=\{r\in R\mid Jr\subseteq I\text{ for some }J\in\GV(R)\}$. Following \cite{WangKim2024}, put
$$
 \tor_{\GV(R)}(M)=\{x\in M\mid Jx=0\text{ for some }J\in\GV(R)\}.
$$
A module is \emph{GV-torsion} if this submodule equals $M$, and \emph{GV-torsion-free} if it is zero. A \emph{$w$-module} is a GV-torsion-free module $M$ such that $\Ext_R^1(R/J,M)=0$ for every $J\in\GV(R)$. For a GV-torsion-free module $M$, its $w$-envelope is
$$
 M_w=\{x\in E(M)\mid Jx\subseteq M\text{ for some }J\in\GV(R)\},
$$
where $E(M)$ is its injective hull. Thus $M$ is a {\em $w$-module} exactly when $M_w=M$. If $R$ is a domain and $M$ is torsion-free in the ordinary sense, then $E(M)=M\otimes_RK$. A GV-torsion-free module $M$ is a $w$-module if and only if there is an exact sequence $0\to M\to N\to C\to0$ with $N$ a $w$-module and $C$ GV-torsion-free \cite[Theorem~6.1.17]{WangKim2024}. Moreover, $M$ is GV-torsion if and only if there is an exact sequence $0\to A\to F\to M\to0$ with $F$ a $w$-module and $A_w=F$ \cite[Theorem~6.2.5]{WangKim2024}.

A GV-torsion-free module $M$ is \emph{of finite type} if $M_w=L_w$ for some finitely generated submodule $L\subseteq M$. For every GV-torsion-free $M$ and $Q\in\wMax(R)$, one has $(M_w)_Q\cong M_Q$ \cite[Theorem~6.2.14]{WangKim2024}.
For a submodule $L$ of a GV-torsion-free module $M$, its relative $w$-closure is
$$
 L^{w,M}=\{x\in M\mid Jx\subseteq L\text{ for some }J\in\GV(R)\}.
$$
Membership is detected by maximal $w$-localizations: $x\in L^{w,M}$ exactly when $x/1\in L_Q$ for every $Q\in\wMax(R)$. In particular, the diagonal map $M\to\prod_{Q\in\wMax(R)}M_Q$ is injective. We use this formulation rather than an intersection of localizations: for a module with ordinary torsion, the individual maps $M\to M_Q$ need not be injective. If $M$ is a $w$-module, the embedded copy of $L_w$ equals $L^{w,M}$; we write it as $L_w$ when the ambient module is clear \cite{KimLimZhou2019,WangKim2024}.

Consider the hereditary torsion theory $\tau_w$ associated to the Gabriel topology $\mathcal F=\{I\subseteq R\mid I_w=R\}$. For a GV-torsion-free module, $r^nM$ is GV-torsion exactly when $r^nM=0$. Consequently its $\tau_w$-secondary condition is precisely the $w$-secondary condition studied here \cite{Garcia2013}.
For a nonzero $w$-secondary module $M$, the ideal $\sqrt{\Ann_R(M)}$ is a prime $w$-ideal; see Proposition~\ref{prop:radical}. The nonzero and secondary hypotheses are essential for this assertion.

We establish localization and $w$-closure criteria for $w$-secondary modules. In general, the nonzero localizations can occur at several primes containing the radical annihilator. Singleton support is available when that radical is a maximal $w$-ideal. Over a P$v$MD we also obtain the corresponding Nagata-ring characterization. We then classify the nonzero-radical $w$-secondary modules over $t$-almost Dedekind domains and retain global strong $w$-representability when every radical in the representation is nonzero. The final characterization uses \emph{local} strong $w$-representability; a local complement is not asserted to descend to a global complement.

\section{$w$-Secondary modules}
This section is devoted to the study of $w$-secondary modules, which are the $w$-operation analogue of classical secondary modules. After establishing their basic properties, we examine localization and $w$-closure and derive results used in the study of $w$-secondary representations.

\begin{definition}\label{def:w-radical}
{\em For an ideal $I\subseteq R$, set
$$
 \wrad{I}=\{r\in R\mid Jr^n\subseteq I
       \text{ for some }n\geq1\text{ and }J\in\GV(R)\}.
$$
Equivalently, $\wrad{I}=\sqrt{I_w}$. In particular, $\sqrt I\subseteq\wrad I$ and $\wrad I$ is a $w$-ideal.
}
\end{definition}

\begin{definition}\label{def:w-support}
{\em The \emph{$w$-support} of an $R$-module $M$ is $\wSupp(M)=\{Q\in\wSpec(R)\mid M_Q\neq0\}$, where $\wSpec(R)$ denotes the set of prime $w$-ideals.
}
\end{definition}

\begin{remark}\label{rem:local-detection}
{\em If $Q$ is a prime $w$-ideal, every $J\in\GV(R)$ meets $R\setminus Q$, so $JR_Q=R_Q$. Hence every GV-torsion module has zero localization at $Q$. Conversely, a module is GV-torsion if its localizations at all maximal $w$-ideals vanish \cite[Theorem~6.2.13]{WangKim2024}.
Every $R_Q$-module is a $w$-module as an $R$-module \cite[Proposition~1.5]{WangKim2015}. For a domain, every prime contained in a prime $w$-ideal is also a $w$-ideal: if $\mathfrak p\subseteq Q$, $Jx\subseteq\mathfrak p$, and $J\in\GV(R)$, choose $j\in J\setminus Q$ and use primality of $\mathfrak p$.
}
\end{remark}

\begin{lemma}\label{lem:scalar-localization}
Let $R$ be a commutative ring, $M$ a GV-torsion-free $R$-module, $Q$ a prime $w$-ideal, and $r\in R$.
\begin{enumerate}
\item If $M/rM$ is GV-torsion, then $M_Q=rM_Q$. If also $r\notin Q$, multiplication by $r/1$ is an automorphism of $M_Q$.
\item If $r\in\wrad{\Ann_R(M)}$, then $(r/1)^nM_Q=0$ for some $n\geq1$; the same exponent works at every prime $w$-ideal.
\end{enumerate}
\end{lemma}
\begin{proof}
The first assertion follows by localizing $M/rM$ and noting that $r/1$ is a unit when $r\notin Q$. For the second, choose $n$ and $J\in\GV(R)$ with $Jr^n\subseteq\Ann_R(M)$. Since $JR_Q=R_Q$, localization gives $(r/1)^nM_Q=0$.
\end{proof}

\begin{definition}\label{def:w-secondary}
{\em Let $R$ be a commutative ring. A nonzero GV-torsion-free module $M$ is called \emph{$w$-secondary} if, for every $r\in R$, either $M/rM$ is GV-torsion or $r^nM=0$ for some $n\geq1$.
}
\end{definition}

\begin{proposition}\label{prop:local-secondary}
Let $M$ be a nonzero GV-torsion-free $w$-secondary module, put $P=\sqrt{\Ann_R(M)}$, and let $Q\in\wSpec(R)$ satisfy $M_Q\neq0$. Then $P\subseteq Q$, $M_Q$ is a classical secondary $R_Q$-module, and
$$
 \sqrt{\Ann_{R_Q}(M_Q)}=PR_Q.
$$
\end{proposition}
\begin{proof}
If $r\in P\setminus Q$, a power of the unit $r/1$ annihilates $M_Q$, a contradiction. Thus $P\subseteq Q$. For $r/s\in R_Q$, the $w$-secondary condition gives either $r^nM_Q=0$ or $rM_Q=M_Q$. Hence $M_Q$ is secondary. If $r\in P$, multiplication by $r/s$ is nilpotent. If $r\notin P$, no power of $r$ annihilates $M$, so $M/rM$ is GV-torsion and multiplication by $r/s$ on $M_Q$ is surjective. A surjective nilpotent endomorphism of a nonzero module is impossible. Therefore the contraction of $\sqrt{\Ann_{R_Q}(M_Q)}$ is exactly $P$, which proves the equality.
\end{proof}

\begin{proposition}\label{prop:radical}
Let $M$ be a nonzero GV-torsion-free $w$-secondary module. Then $I=\Ann_R(M)$ is a primary $w$-ideal and
$$
 P=\sqrt I=\wrad I
$$
is a prime $w$-ideal. In this case, $M$ is also called $P$-$w$-secondary.
\end{proposition}
\begin{proof}
If $Jr\subseteq I$ for $J\in\GV(R)$, then $J(rm)=0$ for every $m\in M$. GV-torsion-freeness gives $rM=0$, so $I_w=I$.

Suppose $ab\in I$ and $a\notin\sqrt I$. Then $M/aM$ is GV-torsion. Multiplication by $b$ factors through this quotient, so $bM$ is both GV-torsion and a submodule of the GV-torsion-free module $M$. Hence $bM=0$, proving that $I$ is primary. Its radical $P$ is therefore prime. Finally, if $Jx\subseteq P$ with $J\in\GV(R)$, finite generation gives $(xJ)^h\subseteq I$ for some $h\geq1$. Since $J^h$ is a GV-ideal \cite[Proposition~6.1.9]{WangKim2024}, GV-torsion-freeness gives $x^hM=0$. Thus $x\in P$, so $P_w=P$. Since $I_w=I$, Definition~\ref{def:w-radical} gives $\wrad I=\sqrt I$.
\end{proof}

\begin{lemma}\label{lem:support}
Let $M$ be a nonzero GV-torsion-free $w$-secondary module and put $P=\sqrt{\Ann_R(M)}$. Then
$$
 \wSupp(M)\subseteq
 \{Q\in\wSpec(R)\mid P\subseteq Q\}.
$$
If $P\in\wMax(R)$, then $\wSupp(M)=\{P\}$.
\end{lemma}
\begin{proof}
The inclusion is Proposition~\ref{prop:local-secondary}. If $P$ is a maximal $w$-ideal, the set on the right is $\{P\}$. Some maximal $w$-localization of $M$ is nonzero; otherwise $M$ would be both GV-torsion and GV-torsion-free. This proves the last assertion.
\end{proof}

The maximality condition cannot be suppressed. For $R=\Z$ and $M=\Q$, one has $w=d$, and $M$ is a $0$-$w$-secondary $w$-module, but $\wSupp(M)=\Spec(\Z)$. A nonzero $w$-secondary module need not remain nonzero at its radical: for a DVR $V$ with quotient field $F$, the divisible module $F/V$ has radical annihilator zero and $(F/V)_{(0)}=0$.

Let $c(f)$ denote the ideal generated by the coefficients of $f\in R[X]$. Put $S_w=\{f\in R[X]\mid c(f)_w=R\}$. The \emph{$w$-Nagata ring} is $R\Nag=R[X]_{S_w}$, and, for an $R$-module $M$, put $M\Nag=M\otimes_RR\Nag\cong S_w^{-1}M[X]$.

\begin{lemma}\label{lem:nagata}
Let $R$ be a commutative ring.
\begin{enumerate}
\item For $P\in\wSpec(R)$, the ideal $P^*=PR\Nag$ is prime and
$$
 (R\Nag)_{P^*}\cong R[X]_{P[X]}\cong R_P(X),
 \qquad
 (M\Nag)_{P^*}\cong M_P\otimes_{R_P}R_P(X).
$$
Here $R_P(X)=R_P[X]_{R_P[X]\setminus PR_P[X]}$ is the ordinary Nagata ring of the local ring $R_P$.
\item If $s\in S_w$ and $s\,a^nM[X]=0$, then $c(s)a^nM=0$, where $c(s)\in\GV(R)$.
\item For every $R$-module $L$, one has $L\Nag=0$ if and only if $L$ is GV-torsion.
\end{enumerate}
\end{lemma}
\begin{proof}
No polynomial in $S_w$ belongs to $P[X]$, since its content cannot be contained in $P$. Thus extension of the prime $P[X]$ to $R[X]_{S_w}$ is prime. Successive localization gives $(R\Nag)_{P^*}=R[X]_{P[X]}$. Localizing the coefficients first gives $R_P[X]_{PR_P[X]}=R_P(X)$. Tensor products give the module isomorphism. These are the usual Nagata localization identities; see \cite{WangKim2024,WangKim2015}.

For (2), apply $s\,a^n$ to a constant polynomial $m\in M$ and compare coefficients. For (3), if $L\Nag=0$ and $x\in L$, some $s_x\in S_w$ satisfies $s_xx=0$ in $L[X]$. Thus $c(s_x)x=0$, so $x$ is GV-torsion. Conversely, the finitely many coefficients of a polynomial in $L[X]$ are annihilated by one GV-ideal, obtained by multiplying their annihilating GV-ideals. A polynomial whose content is that ideal belongs to $S_w$ and kills the given polynomial. Hence $L\Nag=0$.
\end{proof}

In Lemma~\ref{lem:nagata}(3), the annihilating polynomial may depend on the element. No single element of $S_w$ is asserted to annihilate the entire module. Also, $R\to R\Nag$ is flat, but need not be faithfully flat in the ordinary module category.

\begin{lemma}\label{lem:secondary-base-change}
Let $(A,\mathfrak m)$ be a local ring, let $N$ be a nonzero classical secondary $A$-module, and put $\mathfrak p=\sqrt{\Ann_A(N)}$. Let $S\subseteq A[X]$ be multiplicative with $S\cap\mathfrak pA[X]=\varnothing$. Suppose $B=A[X]_S$ is local with maximal ideal $\mathfrak pB$. Then $N\otimes_AB$ is either zero or a $\mathfrak pB$-secondary $B$-module. In particular, it is nonzero when $A\to B$ is faithfully flat.
\end{lemma}
\begin{proof}
Put $N_B=N\otimes_AB$ and suppose $N_B\neq0$. An element outside $\mathfrak pB$ is a unit and acts invertibly on $N_B$. If $b=f/s\in\mathfrak pB$, choose $u\in S$ with $uf\in\mathfrak pA[X]$. The finitely many coefficients of $uf$ generate an ideal $J\subseteq\mathfrak p$. Each coefficient acts nilpotently on $N$, so $J^hN=0$ for some $h\geq1$. Therefore $(uf)^hN_B=0$, and invertibility of $u$ and $s$ gives $b^hN_B=0$. This proves the claim.
\end{proof}

The general characterization must keep a uniform nilpotence exponent for each scalar. The Nagata equivalence below is asserted in the P$v$MD setting, where the needed content identity is available.

\begin{theorem}\label{thm:secondary-characterization}
Let $R$ be a commutative ring and $M$ a nonzero GV-torsion-free $R$-module. The following statements are equivalent.
\begin{enumerate}
\item $M$ is $w$-secondary.
\item For every $r\in R$, either $M_Q=rM_Q$ for every $Q\in\wMax(R)$, or there exists $n\geq1$ such that $r^nM_Q=0$ for every $Q\in\wMax(R)$.
\item $P=\wrad{\Ann_R(M)}$ is a prime $w$-ideal, $M/rM$ is GV-torsion for every $r\notin P$, and for every $r\in P$ there exists $n\geq1$ with $r^nM=0$.
\end{enumerate}
If, in addition, $R$ is a P$v$MD, these conditions are equivalent to
\begin{enumerate}
\item[{\em 4}.] $M\Nag$ is a classical secondary $R\Nag$-module.
\end{enumerate}
Under this additional hypothesis, $\sqrt{\Ann_{R\Nag}(M\Nag)}=PR\Nag$. The localization isomorphism in Lemma~\ref{lem:nagata}(1) holds, but its module may be zero. For an arbitrary $R$, condition~\textup{(4)} still implies \textup{(1)}--\textup{(3)}.
\end{theorem}
\begin{proof}
The equivalence of (1) and (2) follows from Remark~\ref{rem:local-detection}, applied to $M/rM$ and $r^nM$. In the nilpotent case, the GV-torsion submodule $r^nM$ of the GV-torsion-free module $M$ must be zero. Proposition~\ref{prop:radical} gives (1)$\Rightarrow$(3), and (3)$\Rightarrow$(1) is the definition.

Suppose (4) holds. For $r\in R$, either $(M/rM)\Nag=0$ or $(r^nM)\Nag=0$ for some $n\geq1$. Lemma~\ref{lem:nagata}(3) gives respectively that $M/rM$ is GV-torsion or that $r^nM$ is GV-torsion. In the latter case, $r^nM=0$. Thus (1) holds.

For the remaining direction, let $R$ be a P$v$MD and put $T=R\Nag$. We use $fT=c(f)T$ for $f\in R[X]$. Indeed, maximal ideals of $T$ are $QT$ with $Q\in\wMax(R)$, and $T_{QT}=R_Q(X)$. In the valuation domain $R_Q$, the content of $f$ is generated by one coefficient; dividing $f$ by that coefficient gives a primitive polynomial, a unit in $R_Q(X)$. Equality at every maximal ideal proves the identity; see \cite{Kang1989,WangKim2024}.

Assume (1) and set $P=\sqrt{\Ann_R(M)}$. The module $M\Nag$ is nonzero by Lemma~\ref{lem:nagata}(3). If $b=f/s\in PT$, then $f\in P[X]$. Its finitely many coefficients act nilpotently on $M$, so some power of $f$, hence of $b$, annihilates $M\Nag$. If $b\notin PT$, choose a coefficient $a$ of $f$ outside $P$. Since $M/aM$ is GV-torsion, $a(M\Nag)=M\Nag$. The identity $a\in c(f)T=fT$ now gives $f(M\Nag)=M\Nag$. Thus $b$ acts surjectively. This proves (4) and the asserted radical annihilator.
\end{proof}

\begin{corollary}\label{cor:scalar-test}
Let $M$ be a nonzero GV-torsion-free $w$-secondary module, with $P=\sqrt{\Ann_R(M)}$. If $r\notin P$, then $M/rM$ is GV-torsion and multiplication by $r/1$ is an automorphism of $M_P$. If $r\in P$, then $r^nM=0$ and $(r/1)^nM_P=0$ for some $n\geq1$.
\end{corollary}
\begin{proof}
Apply Theorem~\ref{thm:secondary-characterization}(3) and Lemma~\ref{lem:scalar-localization}.
\end{proof}

\begin{remark}\label{rem:classical-secondary}
{\em Every GV-torsion-free classical secondary module is $w$-secondary. Conversely, each nonzero localization of a $w$-secondary module at a prime $w$-ideal is classically secondary, by Proposition~\ref{prop:local-secondary}. This does not imply that there is a unique nonzero localization or that localization at the radical is nonzero.
}
\end{remark}

\begin{theorem}\label{thm:envelope-maximal-localization}
Let $R$ be a commutative ring and $M$ a nonzero GV-torsion-free module. Then $M$ is $w$-secondary if and only if $M_w$ is $w$-secondary, and $\Ann_R(M)=\Ann_R(M_w)$.

Moreover, for a fixed $P\in\wMax(R)$, the following are equivalent.
\begin{enumerate}
\item $M$ is $w$-secondary and $\sqrt{\Ann_R(M)}=P$.
\item $M_w$ is $w$-secondary and $\sqrt{\Ann_R(M_w)}=P$.
\item $M_P$ is a nonzero $PR_P$-secondary module and $M_Q=0$ for every $Q\in\wMax(R)\setminus\{P\}$.
\item $\wSupp(M)=\{P\}$ and $M_P$ is a $PR_P$-secondary module.
\end{enumerate}
\end{theorem}
\begin{proof}
The injective hull of a GV-torsion-free module is GV-torsion-free \cite[Theorem~6.2.2]{WangKim2024}; hence so is $M_w$. If $aM=0$ and $x\in M_w$, choose $J\in\GV(R)$ with $Jx\subseteq M$. Then $Jax=0$, whence $ax=0$. This proves equality of the annihilators. Also, $(M_w)_Q\cong M_Q$ at every maximal $w$-ideal \cite[Theorem~6.2.14]{WangKim2024}. Theorem~\ref{thm:secondary-characterization}(2) now proves the first equivalence.

Consequently (1) and (2) are equivalent. Lemma~\ref{lem:support} and Proposition~\ref{prop:local-secondary} give (1)$\Rightarrow$(4), and (4)$\Rightarrow$(3) is immediate. Assume (3). If $r\notin P$, all maximal $w$-localizations of $M/rM$ vanish. If $r\in P$, some $r^n$ annihilates $M_P$, and it also annihilates every other maximal $w$-localization, which is zero. Thus $r^nM=0$. This proves $w$-secondaryness. The nonzero $PR_P$-secondary localization ensures that its global radical annihilator is exactly $P$, proving (1).
\end{proof}

We next give the $w$-analogue of the elementary bounded-primary-module criterion \cite[Lemma~2.1]{Pakyari2020}.
\begin{lemma}\label{lem:bounded-secondary}
Let $R$ be a commutative ring and let $P$ be a nonzero maximal $w$-ideal. If $M$ is a nonzero GV-torsion-free module with $P^nM=0$ for some $n\geq1$, then $M$ is $w$-secondary.
\end{lemma}
\begin{proof}
If $r\in P$, then $r^nM=0$. If $r\notin P$, maximality gives $(P+rR)_w=R$. Hence
$$
 R=((P+rR)^n)_w\subseteq(P^n+rR)_w\subseteq R.
$$
Choose $J\in\GV(R)$ with $J\subseteq P^n+rR$. Then $JM\subseteq rM$, so $M/rM$ is GV-torsion.
\end{proof}

The radical version does not require a uniformly annihilating power of $P$.
\begin{lemma}\label{lem:maximal-radical}
Let $R$ be an integral domain and $M$ a nonzero $w$-module. If $P=\sqrt{\Ann_R(M)}$ is a maximal $w$-ideal, then $M$ is $w$-secondary.
\end{lemma}
\begin{proof}
For $r\in P$, some power of $r$ annihilates $M$. If $r\notin P$, multiplication by $r$ is invertible on $M_P$. For $Q\neq P$ maximal $w$, choose $x\in P\setminus Q$. A power of $x$ annihilates $M$, so $M_Q=0$. Thus every maximal $w$-localization of $M/rM$ is zero, and $M/rM$ is GV-torsion.
\end{proof}

\begin{corollary}\label{cor:secondary-submodule}
Let $R$ be an integral domain and $M$ a nonzero $w$-module with $P=\sqrt{\Ann_R(M)}\in\wMax(R)$. Every nonzero $w$-submodule $N\subseteq M$ is $w$-secondary with $\sqrt{\Ann_R(N)}=P$.
\end{corollary}
\begin{proof}
For every nonzero GV-torsion-free module, its annihilator and the radical of its annihilator are $w$-ideals, by the corresponding parts of the proof of Proposition~\ref{prop:radical}. Thus $P\subseteq\sqrt{\Ann_R(N)}\subsetneq R$, and maximality gives equality. Apply Lemma~\ref{lem:maximal-radical}.
\end{proof}

The next two results are the $w$-analogues of \cite[Propositions~2.4--2.5]{Pakyari2020}.
\begin{proposition}\label{prop:localized-action}
Let $R$ be an integral domain and $M$ a nonzero $w$-module with $P=\sqrt{\Ann_R(M)}\in\wMax(R)$. Then $M$ admits a canonical $R_P$-module structure.
\end{proposition}
\begin{proof}
Fix $s\notin P$. At $P$, multiplication by $s$ is invertible, and at each other maximal $w$-ideal both source and target are zero, by the proof of Lemma~\ref{lem:maximal-radical}. Thus multiplication by $s$ on $M$ has GV-torsion kernel and cokernel. Its kernel is zero because $M$ is GV-torsion-free. Its image is isomorphic to the $w$-module $M$, hence is a $w$-submodule of $M$. The quotient by such a submodule is GV-torsion-free \cite[Theorem~6.1.17]{WangKim2024}; being also GV-torsion, it is zero. Thus every $s\notin P$ acts bijectively, giving the canonical $R_P$-action.
\end{proof}

\begin{proposition}\label{prop:localization-isomorphism}
Under the hypotheses of Proposition~\ref{prop:localized-action}, the canonical map $M\to M_P$ is an isomorphism. In particular, $M\cong(M_P)_w=M_P$.
\end{proposition}
\begin{proof}
Proposition~\ref{prop:localized-action} makes the canonical localization map an isomorphism. The module $M_P$ is a $w$-module over $R$ by \cite[Proposition~1.5]{WangKim2015}, so $(M_P)_w=M_P$.
\end{proof}

\begin{lemma}\label{lem:direct-sums}
Let $R$ be a commutative ring and $M_1,\ldots,M_n$ GV-torsion-free modules. Then
$$
 (M_1\oplus\cdots\oplus M_n)_w
 \cong(M_1)_w\oplus\cdots\oplus(M_n)_w.
$$
The same statement holds for an arbitrary family of GV-torsion-free modules.
\end{lemma}
\begin{proof}
For a finite family, $E(\bigoplus_iM_i)=\bigoplus_iE(M_i)$ \cite{Lam2012}. If $x_iJ_i\subseteq M_i$ with $J_i\in\GV(R)$, the GV-ideal $\prod_iJ_i$ carries $(x_1,\ldots,x_n)$ into $\bigoplus_iM_i$. Conversely, an annihilating GV-ideal for the tuple works in every coordinate. This proves the finite assertion.
For the arbitrary assertion, first note that a direct sum of $w$-modules is a $w$-module. It is GV-torsion-free, and any homomorphism from a finitely generated GV-ideal to the sum has image in a finite subsum and therefore extends to $R$. This is the defining vanishing of $\Ext_R^1(R/J,-)$. Now $\bigoplus_iM_i$ embeds in the $w$-module $\bigoplus_i(M_i)_w$ with GV-torsion cokernel. The latter is its $w$-envelope \cite{WangKim2024,WangKim2015}.
\end{proof}

A classical secondary representation is an expression $M=N_1+\cdots+N_t$ with each $N_i$ secondary. It is \emph{minimal} if the radical annihilators are pairwise distinct and no component is redundant. The radicals form the set $\Att(M)$ of attached prime ideals, independently of the chosen minimal classical representation \cite{Macdonald1973}. A module is \emph{strongly representable} if it is a finite direct sum of secondary submodules. Clearly every strongly representable module is representable.

For a hereditary torsion theory $\tau$ with Gabriel filter $L(\tau)$, the $\tau$-closure of $N\subseteq M$ is
$$
 \Cl_\tau^M(N)=\{x\in M\mid Ix\subseteq N\text{ for some }I\in L(\tau)\}.
$$
The submodule is {\em $\tau$-closed} if it equals its closure, and {\em $\tau$-dense} if its closure is $M$. Equivalently, these conditions mean that $M/N$ is $\tau$-torsion-free or $\tau$-torsion, respectively. For $\tau_w$, this is the relative $w$-closure already defined; if $M$ is GV-torsion-free, it agrees with the embedded copy of $N_w\cap M$. In particular a GV-torsion-free module and a $\tau_w$-dense submodule have the same $w$-envelope \cite[Lemma~2.1]{KimLimZhou2019}.

Following the torsion-theoretic convention of \cite{Garcia2013}, a \emph{$w$-secondary representation} of a $w$-module $M$ is
$$
 M=(N_1+\cdots+N_n)_w,
$$
where each $N_i$ is $w$-secondary. Such a representation is called \emph{minimal} here when the radicals $P_i=\sqrt{\Ann_R(N_i)}$ are pairwise distinct; $M$ is then \emph{$w$-representable}. If $M=(N_1\oplus\cdots\oplus N_n)_w$ with distinct radicals, it is \emph{strongly $w$-representable}. This convention for minimality does not assert irredundancy; no uniqueness of the displayed radicals is used below. We refer to these ideals as the radical annihilators of the components, not as associated primes of $M$. In particular, $N_j\cap\sum_{i\neq j}N_i=0$ for all $j$ gives a strong $w$-representation.

Two ideals $I,J$ are \emph{$w$-comaximal} if $(I+J)_w=R$. The following is the $w$-analogue of \cite[Lemma~4.5]{Pakyari2020}.
\begin{proposition}\label{prop:comaximal}
Let $R$ be a commutative ring and let $M$ be a $w$-module with a $w$-secondary representation whose radical annihilators are pairwise $w$-comaximal. Then the displayed components have direct sum and $M$ is strongly $w$-representable.
\end{proposition}
\begin{proof}
Write $M=(N_1+\cdots+N_n)_w$ and $I_i=\Ann_R(N_i)$, so $P_i=\sqrt{I_i}$. If $(I_i+I_j)_w$ were proper, it would be contained in a maximal $w$-ideal, which would contain both $P_i$ and $P_j$, a contradiction. Thus the $I_i$ are pairwise $w$-comaximal. The same prime-ideal argument gives $(I_j+\bigcap_{i\neq j}I_i)_w=R$. Choose $J\in\GV(R)$ contained in this sum. If $x\in N_j\cap\sum_{i\neq j}N_i$, then both $I_jx=0$ and $(\bigcap_{i\neq j}I_i)x=0$. Hence $Jx=0$, and GV-torsion-freeness gives $x=0$. Thus $\sum_iN_i=\bigoplus_iN_i$.
\end{proof}

\begin{corollary}\label{cor:dimension-one-nonzero}
Let $R$ be an integral domain with $\wdim R=1$. Every $w$-representable $w$-module admitting a representation with only nonzero radical annihilators is strongly $w$-representable.
\end{corollary}
\begin{proof}
All the displayed prime $w$-ideals are maximal $w$-ideals. Apply Proposition~\ref{prop:comaximal}.
\end{proof}

\begin{proposition}\label{prop:restriction}
Let $R$ be an integral domain such that every $w$-representable $w$-module is strongly $w$-representable. Then, for each nonzero \emph{prime $w$-ideal} $P$, every representable $R_P$-module is strongly representable.
\end{proposition}
\begin{proof}
Let $L$ be a representable $R_P$-module. The module $L$ and each of its $R_P$-secondary components are $w$-modules over $R$ \cite[Proposition~1.5]{WangKim2015}. Restriction of scalars preserves classical secondaryness \cite[Lemma~2.4]{PakyariRMJ2023}; hence $L$ is $w$-representable over $R$.

By hypothesis, choose a strong $w$-representation of $L$, not necessarily with the original components. Taking $w$-envelopes of its components and using Lemma~\ref{lem:direct-sums} gives an ordinary direct sum $L=A_1\oplus\cdots\oplus A_t$ of $w$-secondary $w$-modules. Each $A_i$ is an $R_P$-submodule. Indeed, for $x\in A_i$ and $s\notin P$, decompose $s^{-1}x=\sum_jx_j$ in this direct sum. Then $sx_j=0$ for $j\neq i$, so $x_j=0$ because $s$ acts injectively on $L$. Thus $s^{-1}x\in A_i$.

For a scalar $r\in R$, the $w$-secondary alternatives for $A_i$ become the classical alternatives after localization at $P$, since $A_i=(A_i)_P$ and every GV-torsion module vanishes there. Hence every $A_i$ is classically secondary over $R_P$, proving the assertion.
\end{proof}

\section{On $t$-almost Dedekind domains}
Following \cite{Kang1989}, a nonfield integral domain $R$ is \emph{$t$-almost Dedekind} if $R_P$ is a rank-one discrete valuation domain (DVR) for every $P\in\tMax(R)=\wMax(R)$. These domains admit ideal-theoretic, ring-theoretic, and homological characterizations; see \cite{Kang1989,Zhou2021}. By \cite[Theorem~4.5]{Kang1989}, a $t$-almost Dedekind domain is a {\em Pr\"ufer $v$-multiplication domain} (P$v$MD), meaning that every nonzero finitely generated ideal is $t$-invertible. In particular $t=w$, and $\wdim R=1$.

The following is the $w$-analogue of the bounded secondary-module classification over almost Dedekind domains \cite[Theorem~2.7]{Pakyari2020}.
\begin{proposition}\label{prop:classification}
Let $R$ be a nonfield $t$-almost Dedekind domain. Let $M$ be a nonzero $w$-secondary $w$-module with $P=\sqrt{\Ann_R(M)}\neq0$. Then
$$
 M\cong\bigoplus_{i\in\Gamma}\bigl(R/(P^{n_i})_t\bigr)_w,
$$
where $\Gamma\neq\varnothing$, every $n_i\geq1$, and the set of exponent values is finite.
\end{proposition}
\begin{proof}
The prime $w$-ideal $P$ is maximal because $\wdim R=1$. Propositions~\ref{prop:local-secondary} and \ref{prop:localization-isomorphism} identify $M$ with the $PR_P$-secondary module $M_P$. Write $PR_P=aR_P$. Some power $a^h$ annihilates $M_P$. The classical structure theorem for arbitrary bounded torsion modules over a DVR gives
$$
 M_P\cong\bigoplus_{i\in\Gamma}R_P/P^{n_i}R_P,
 \qquad 1\leq n_i\leq h;
$$
see \cite{Kaplansky1952} and \cite[Proposition~2.2]{Pakyari2020}. No finite-generation hypothesis on $M$ is needed for this bounded-module theorem.

For $n\geq1$, put $I=(P^n)_t=(P^n)_w$ and $C=R/I$. Since $I$ is $w$-closed, $C$ is GV-torsion-free. Also $P^n\subseteq I\subseteq P$ and $P^nC=0$, so $C$ is $P$-$w$-secondary by Lemma~\ref{lem:bounded-secondary}. At $P$, $C_P\cong R_P/P^nR_P$, whereas $C_Q=0$ at every other maximal $w$-ideal. The canonical map $C\to C_P$ has zero kernel and GV-torsion cokernel: its kernel vanishes at every maximal $w$-ideal and is a submodule of $C$. Since $C_P$ is a $w$-module, it is the $w$-envelope of $C$ \cite{WangKim2024,WangKim2015}. Thus $C_w\cong R_P/P^nR_P$. Substituting these isomorphisms proves the result.
\end{proof}

The inner $t$-closure cannot simply be omitted: $R/P^n$ need not be GV-torsion-free. Since $t=w$ on a P$v$MD, one may equally write $(R/(P^n)_w)_w$; the outer closure is the module $w$-envelope.

\begin{corollary}\label{cor:finite-type}
Let $R$ be a nonfield $t$-almost Dedekind domain, $M$ a nonzero $w$-module of finite type, and $P$ a nonzero prime $w$-ideal. Then $M$ is $w$-secondary with radical annihilator $P$ if and only if, for every $x\in M$, there is $n\geq1$ such that $P^nx=0$.
\end{corollary}
\begin{proof}
In the forward direction, Proposition~\ref{prop:classification} gives a single $h$ with $P^hM=0$. Conversely, write $M=(m_1R+\cdots+m_kR)_w$. Choose $h$ such that $P^hm_j=0$ for every $j$. For $x\in M$, choose $J\in\GV(R)$ with $Jx\subseteq\sum_jm_jR$. Then $J(P^hx)=0$, so $P^hx=0$ by GV-torsion-freeness. Hence $M$ is $w$-secondary by Lemma~\ref{lem:bounded-secondary}, and its proper radical annihilator is $P$ by maximality.
\end{proof}

The next result retains the global decomposition theorem when all component radicals are nonzero; compare \cite[Theorem~4.6]{Pakyari2020}.
\begin{theorem}\label{thm:nonzero-radicals}
Let $R$ be a $t$-almost Dedekind domain and let $M$ be a nonzero $w$-module. The following are equivalent.
\begin{enumerate}
\item $M$ admits a minimal $w$-secondary representation with only nonzero radical annihilators.
\item $M$ admits a strong $w$-representation with only nonzero radical annihilators.
\item There are pairwise distinct nonzero prime $w$-ideals $P_1,\ldots,P_t$ such that
$$
 M\cong\bigoplus_{i=1}^n
       \left(\bigoplus_{j\in\Gamma_i}
       \bigl(R/(P_i^{n_{ij}})_t\bigr)_w\right),
$$
where each $\Gamma_i$ is nonempty and each set $\{n_{ij}\mid j\in\Gamma_i\}$ is a finite set of positive integers.
\end{enumerate}
\end{theorem}
\begin{proof}
For (1)$\Rightarrow$(2), all the displayed radical annihilators are maximal $w$-ideals, so Proposition~\ref{prop:comaximal} applies. For (2)$\Rightarrow$(3), take the $w$-envelopes of the direct components. They remain secondary by Theorem~\ref{thm:envelope-maximal-localization}, and Lemma~\ref{lem:direct-sums} gives an ordinary finite direct sum of $w$-modules. Apply Proposition~\ref{prop:classification} to each. Conversely, for each $i$ the inner sum is a nonzero $w$-module annihilated by a fixed power of $P_i$, so it is $P_i$-$w$-secondary by Lemma~\ref{lem:bounded-secondary}. This gives (3)$\Rightarrow$(1).
\end{proof}

A zero-radical component requires a different argument. The following proposition asserts a local complement, not a direct summand of the prescribed bounded component and not an automatically descended submodule.

\begin{proposition}\label{prop:local-splitting}
Let $R$ be a $t$-almost Dedekind domain and let
$$
 M=(K_1+\cdots+K_n)_w
$$
be a minimal $w$-secondary representation with $\sqrt{\Ann_R(K_1)}=0$ and $P_i=\sqrt{\Ann_R(K_i)}\neq0$ for $i\geq2$. Then:
\begin{enumerate}
\item For each $i\geq2$, there is an $R_{P_i}$-submodule $B_i\subseteq M_{P_i}$ such that
$$
 M_{P_i}=(K_1)_{P_i}\oplus B_i.
$$
The module $B_i$ is either zero or $P_iR_{P_i}$-secondary. It can be written $B_i=(1-\alpha_i)((K_i)_{P_i})$, where $\alpha_i:(K_i)_{P_i}\to(K_1)_{P_i}$ restricts to the identity on their intersection.
\item If $Q\in\wMax(R)\setminus\{P_2,\ldots,P_n\}$, then $M_Q=(K_1)_Q$, which is either zero or divisible over $R_Q$.
\item Every nonzero $M_Q$, $Q\in\wMax(R)$, is strongly representable over $R_Q$.
\end{enumerate}
\end{proposition}
\begin{proof}
For $j\geq2$, Lemma~\ref{lem:support} gives $(K_j)_Q=0$ unless $Q=P_j$. Also $(K_1)_Q$ is divisible: for $0\neq r\in R$, the quotient $K_1/rK_1$ is GV-torsion because $\Ann_R(K_1)=0$. Localization gives $r(K_1)_Q=(K_1)_Q$.

Fix $i\geq2$ and put $V=R_{P_i}$, $D=(K_1)_{P_i}$, and $T=(K_i)_{P_i}$. Then $M_{P_i}=D+T$, and $D$ is divisible over the DVR $V$, hence injective by Baer's criterion \cite{WangKim2024,Lam2012}. Therefore the inclusion $D\cap T\hookrightarrow D$ extends to $\alpha_i:T\to D$. The image $B_i=(1-\alpha_i)T$ satisfies $D\cap B_i=0$ and $D+B_i=D+T$. It is zero or a secondary image of the $P_iV$-secondary module $T$. This proves (1). The support calculation gives (2), and (3) follows from (1) and (2), omitting zero summands.
\end{proof}

We call a $w$-module $M$ \emph{locally strongly $w$-representable} if $M_Q$ is zero or strongly representable over $R_Q$ for every $Q\in\wMax(R)$. This is a condition on localized modules, not on the descent of their chosen summands.

\begin{theorem}\label{thm:main}
Let $R$ be a nonfield integral domain. Then $R$ is a $t$-almost Dedekind domain if and only if $\wdim R=1$ and every $w$-representable $w$-module is locally strongly $w$-representable.

In particular, the condition that $\wdim R=1$ and every $w$-representable $w$-module is globally strongly $w$-representable still implies that $R$ is $t$-almost Dedekind.
\end{theorem}
\begin{proof}
Suppose $R$ is $t$-almost Dedekind. Then $\wdim R=1$ by \cite[Theorem~4.5]{Kang1989}. If a $w$-representation has only nonzero radicals, Theorem~\ref{thm:nonzero-radicals} gives a strong global representation, whose nonzero localizations are strong. If a zero radical occurs together with other radicals, use Proposition~\ref{prop:local-splitting}. A single zero-radical component has divisible, hence secondary or zero, localizations. Thus every $w$-representable $w$-module is locally strongly $w$-representable.

Conversely, let $\wdim R=1$ and assume the stated local representability condition. Fix $P\in\wMax(R)$. Then $P\neq0$ and $\dim R_P=1$, by Remark~\ref{rem:local-detection}. Every representable $R_P$-module $L$, together with its classical secondary components, is a $w$-module over $R$ by \cite[Proposition~1.5]{WangKim2015}. Restriction of scalars makes those components $w$-secondary \cite[Lemma~2.4]{PakyariRMJ2023}. Hence $L$ is a $w$-representable $w$-module over $R$. By hypothesis, $L=L_P$ is strongly representable over $R_P$. The local one-dimensional converse \cite[Theorem~2.7]{PakyariRMJ2023} now implies that $R_P$ is Dedekind, hence a DVR. Since $P$ was arbitrary, $R$ is $t$-almost Dedekind.

For the last assertion, a finite strong $w$-representation localizes to a direct sum of classical secondary modules or zero modules by Proposition~\ref{prop:local-secondary}. Hence global strong $w$-representability implies the local condition just proved sufficient.
\end{proof}

Theorem~\ref{thm:main} does not assert that every $w$-representable module over an arbitrary $t$-almost Dedekind domain is globally strongly $w$-representable. Proposition~\ref{prop:local-splitting} constructs its correction maps only after localization. The global conclusion with nonzero component radicals remains Theorem~\ref{thm:nonzero-radicals}. In particular, the local converse in \cite[Theorem~2.7]{PakyariRMJ2023}, rather than the unrestricted global assertion in its Theorem~2.9, is the result used above.

\end{document}